\documentclass[11pt]{article}
\usepackage{amssymb}
\usepackage{amsfonts}
\usepackage{amsmath}

\newtheorem{theorem}{Theorem}
\newtheorem{acknowledgement}[theorem]{Acknowledgement}

\newtheorem{lemma}[theorem]{Lemma}

\newenvironment{proof}[1][Proof]{\noindent\textbf{#1.} }{\ \rule{0.5em}{0.5em}}
\input{tcilatex}
\begin{document}

\author{Ismail Naci Cangul, Gokhan Soydan, Yilmaz Simsek}
\title{A p-ADIC LOOK AT THE DIOPHANTINE EQUATION $x^{2}+11^{2k}=\allowbreak
y^{n}$}
\date{}
\maketitle

\begin{abstract}
We find all solutions of Diophantine equation $x^{2}+11^{2k}=\allowbreak
y^{n},$ $x\geq 1,$ $y\geq 1,$ $k\in 
\mathbb{N}
,$ $n\geq 3$. We give $p$-adic interpretation of this equation.



\smallskip
\end{abstract}

\section{\protect\bigskip {\protect\normalsize Introduction}}

In this paper, we consider the equation%
\begin{equation}
x^{2}+11^{2k}=y^{n},\text{ }x\geq 1,\text{ }y\geq 1,\text{ }k\geq 1,\text{ }%
n\geq 3.  \tag*{(1.1)}  \label{1.1}
\end{equation}%
Our main result is the following.

\begin{theorem}
Equation \ref{1.1} has only one solution%
\begin{equation*}
n=3\text{ \ \ \textit{and \ }}(x,y,k)=(2\cdot 11^{3\lambda },5\cdot
11^{2\lambda },1+3\lambda )
\end{equation*}%
where $\lambda \geq 0$ is any integer.
\end{theorem}

\section{\protect\normalsize Reduction to Primitive Solution}

\qquad Note that it sufficies to study \ref{1.1} when $gcd(x,y)=1$. Such
solutions are called \textit{primitive.} Let $(x,y,k,n)$ be a non primitive
solution. Let $x=11^{a}\cdot x_{1}$, $y=11^{b}\cdot y_{1}$ with $a\geq
1,~b\geq 1$ and $11\nmid x_{1}y_{1}$. \ref{1.1} becomes%
\begin{equation}
11^{2a}x_{1}^{2}+11^{2k}=11^{nb}y_{1^{{}}}^{n}.  \tag*{(2.1)}  \label{2.1}
\end{equation}%
We have either $2k=nb\leq 2a$ or $2a=nb<2k$. First case leads to $%
X^{2}+1=Y^{n},$ $X=11^{a-k}x_{1}$ and $Y=y_{1},$ which has no solution by
Lebesgue's result, and second leads to $X^{2}+11^{2k_{1}}=Y^{n},X=x_{1},$ $%
Y=y_{1}$ and $2k_{1}=2k-2a=2k-nb$. $(X,Y,k_{1},n)$ is a solution of \ref{1.1}
and a primitive solution is $(2,5,1,3).$ If $(x_{1},y_{1},k_{1},n)=(2,5,1,3)$%
, then $2k=2+2a=2+3b$ and hence $a=3\lambda $ and $b=2\lambda $ for $\lambda
\in N.$ Now $(x,y,k,n)=(2\cdot 11^{3\lambda },5\cdot 11^{2\lambda
},1+3\lambda ,3)$. It remains to prove that the only primitive solution is
indeed $(2,5,1,3).$

\section{\protect\normalsize The Case When n=3}

\begin{lemma}
\label{lemma (n=3)} The only primitive solution of \ref{1.1} with $n=3$ is $%
(2,5,1).$
\end{lemma}

\begin{proof}
As $x$ and $y$ are coprime and $11^{2k}\equiv 1$ $\pmod 4$, we get $x$ is
even in 
\begin{equation}
(x+i11^{k})(x-i11^{k})=y^{3}.  \tag*{(3.1)}  \label{3.1}
\end{equation}%
Hence $x+i11^{k}$ and $x-i11^{k}$ are coprime in $%
\mathbb{Z}
\lbrack i]$ which is a UFD. As the only units of $%
\mathbb{Z}
\lbrack i]$ are $\pm 1,\pm i$, we get%
\begin{equation}
x+i11^{k}=(u+iv)^{3};~~x-i11^{k}=(u-iv)^{3}.  \tag*{(3.2)}  \label{3.2}
\end{equation}%
Eliminating $x$, we get $2i11^{k}=(u+iv)^{3}-(u-iv)^{3}~$or $%
11^{k}=v(3u^{2}-v^{2})$. Note that $u$ and $v$ are coprime since otherwise
any prime factor of $u$ and $v$ will also divide both $x$ and $y$. Therefore 
$v=\pm 1$ or $v=\pm 11^{k}$, which lead to%
\begin{equation}
3u^{2}=1\pm 11^{k},~3u^{2}=\pm 1+11^{2k},  \tag*{(3.3)}  \label{3.3}
\end{equation}%
respectively. First equation is impossible as if the sign is $-$, then right
hand side is negative, while if the sign is $+$ and $k$ is even, then right
hand side is congruent to $2$ modulo $3$ while left hand side is divisible
by $3$. Finally if the sign is $+$ and $k$ is odd, this equation has only
one solution. Let's write $m=(k-1)/2$, $X=3u$, $Y=11^{m}.$Then the equation
becomes a Pell equation with an additional condition, namely $%
X^{2}-33Y^{2}=3,$ with $Y=11^{m}.$ Then $X+\sqrt{33}Y=(6+\sqrt{33})(23+4%
\sqrt{33})^{r}$. So $Y=\pm y_{r},$ where $(y_{r})$ is given by $%
y_{-1}=-1,~y_{0}=1,~y_{r+1}=46y_{r}-y_{r-1}.~$This sequence is symmetric
about $r=-1,2.$ As we are interested in $y_{r}=\pm 11^{m},~$we look at the
sequence in modulo $11$: $-1,1,3,5,-4,-2,0,2,4,-5,-3,-1,1...$, with a period
of length $11$. Thus $11|y_{r}$ if and only if $r\equiv 5~(\func{mod}11).~$%
But any other prime that divides $y_{5}=210044879$ will also divide any $%
y_{r}$ with $r\equiv 5~(\func{mod}11)$. As $y_{5}=210044879=11.373.51193,$
we find that $r\equiv 5~(\func{mod}11)~$implies $373|y_{r}~$and $51193|y_{r}$%
. Thus $m=0$ is the only possibility for $y_{r}=\pm 11^{m}$. From here,$\
u=\pm 2,$ $v=1,k=1$ and so $(x,y,k)=(2,5,1).$ For the second equation, the
sign must be $-1$. Thus $(11^{k})^{2}-3u^{2}=1$. $X^{2}-3Y^{2}=1$ has a
smallest solution $(X_{1},Y_{1})=(2,1)$. Furthermore $(X_{2},Y_{2})=(7,4)$
and $(X_{3},Y_{3})=(26,15).$ $(X_{m})$ is a Lucas sequence of second type.
By Primitive Divisor Theorem, \cite{Carmichael}, if $m>12$, then $X_{m}$ has
a prime factor $p\equiv 1~\pmod m$. In particular, $X_{m}$ can not be a
power of $11$ if $m>12$. One can check that $m\leq 12$ such that $X_{m}$ can
not be a power of $11$.
\end{proof}

\section{\protect\normalsize The Case\ When n=4}

\begin{lemma}
\label{lemma (n=4)}Equation \ref{1.1} has no solution for $n=4$.
\end{lemma}

\begin{proof}
Now we rewrite equation \ref{1.1} as $11^{2k}=(y^{2}+x)(y^{2}-x)$. Since $x$
is even and $y$ is odd, we have that $y^{2}+x$ and $y^{2}-x$ are coprime.
Thus%
\begin{equation}
y^{2}-x=1;~~~y^{2}+x=11^{2k},  \tag*{(4.1)}  \label{4.1}
\end{equation}%
which leads to $(11^{k})^{2}-2y^{2}=-1$. \ Equation \ref{4.1} gives a
solution $(X,Y)$ to Pell equation $X^{2}-2Y^{2}=\pm 1$ with $X=11^{k}$ and $%
Y=y$. The first solution of equation \ref{4.1} is $(X_{1},Y_{1})=(1,1)$.
Further $X_{2}=3$, $X_{3}=7$ and $X_{4}=17$. By checking $X_{m}$ for all $%
\leq 12$ and invoking the Prime Divisor Theorem for $m>12$, we get that $%
X_{m}$ can not be a power of $11$.
\end{proof}

\section{\protect\normalsize The Remaining Cases}

If $(x,y,k,n)$ is a primitive solution to \ref{1.1} and $d>2$ divides $n$,
then $(x,y^{n/d},k,d)$ is also a primitive solution of \ref{1.1}. Since $%
n\geq 3$ is coprime to $3$ and not a multiple of $4,$ there is a prime $%
p\geq 5$ dividing $n$. Replace $n$ by this prime. Look again at $%
(x+i11^{k})(x-i11^{k})=y^{p}$. Since $x$ is even and $y$ is odd, we get that 
$x+11^{k}i$ and $x-11^{k}i$ are coprime in $%
\mathbb{Z}
\lbrack i]$. Then there exist $u$ and $v$ so that if $\alpha =u+iv$, then $%
x+i11^{k}=\alpha ^{p}~$and$~x-i11^{k}=\overset{\_}{\alpha }^{p}.~$Hence%
\begin{equation}
\frac{11^{k}}{v}=\frac{\alpha ^{p}-\overset{\_}{\alpha }^{p}}{\alpha -%
\overset{\_}{\alpha }}\in 
\mathbb{Z}
\text{. }  \tag*{(5.1)}  \label{5.1}
\end{equation}%
$u_{n}=(\alpha ^{n}-\overset{\_}{\alpha }^{n})/(\alpha -\overset{\_}{\alpha }%
)$ for all $n\geq 0$ is a Lucas sequence. A prime factor $q$ of $u_{n}$ is
called \textit{primitive} if $q\nmid u_{n}$ for any $0<k<n$ and $q\nmid
(\alpha -{\overline{\alpha }})^{2}=-4v^{2}$. If~such a $q~$exists, then $%
q\equiv \pm 1\pmod n$, where the sign coincides with the Legendre symbol $%
(-1\mid q)$. By $\cite{Bilu}$, we know that if $n\geq 5$ is prime, then $%
u_{n}$ always has a prime factor except for finitely many \textit{%
exceptional triples~}$(\alpha ,{\overline{\alpha }},n)$, and all of them
appear in the Table 1 in $\cite{Bilu}.$

Let $u_{n}~$be without a primitive divisor. Table 1 reveals that there is 
\textit{no defective} Lucas number $u_{n}$ with roots $\alpha ,\overset{\_}{%
\alpha }$ in$~%
\mathbb{Z}
\lbrack i]$.~

Since $n\geq 5$ is prime, it follows that $11$ is primitive for $u_{n}$.
Thus $11\equiv +1\pmod 5$. But since $(-1\mid q)=-1,~$then $11$ can't be a
primitive divisor. Thus, there are no more primitive solutions to our
equation.

\section{\protect\normalsize Further Remarks and Observations}

The Dirichlet $L$-functions relate certain Euler products to various objects
such as Diophantine equations, representations of Galois group, Modular
forms etc. These functions play a crucial role not only in complex analysis
but also in number theory. The $p$-adic $L$-function agrees with the
Dirichlet $L$-functions at negative integers. $p$-adic $L$-function can be
used to prove congruences for generalized Bernoulli numbers. It is
well-known that following Diophantine equation is related to Bernoulli
polynomials $B_{n}(x)$ 
\begin{equation}
aB_{n}^{{}}(x)=bB_{m}(x)+C(y)\text{, }a,b\in \mathbb{Q}\setminus \left\{
0\right\}  \tag*{(6.1)}  \label{6.1}
\end{equation}%
with $n\geq m>\deg (C)+2~$and for a rational polynomial $C(y)$.

Following are some open problems: How can we generalize such a Diophantine
equation to twisted Bernoulli, Euler and generalized Bernoulli polynomials
attached to Dirichlet character? What is the relation between \ref{6.1}, $p$%
-adic $L$-function and Kummer congruences for Bernoulli numbers? How can one
determine cyclotomic units of \ref{1.1} and Lemma \ref{lemma (n=3)}? Are
there relations between Lucas, Lehmer, Bernoulli and Euler numbers, and \ref%
{6.1}.

\begin{acknowledgement}
The paper is supported by Uludag Univ. Research Fund, Projects 2008-31 and
2008-51, and Akdeniz Univ. Administration.
\end{acknowledgement}

\bigskip

Ismail Naci Cangul, Uludag Univ., Bursa, Turkey, cangul@uludag.edu.tr

Gokhan Soydan, Isiklar High School, Bursa, TURKEY, gsoydan@uludag.edu.tr

Yilmaz Simsek, Akdeniz Univ., Antalya, Turkey, ysimsek@akdeniz.edu.tr

\end{document}